\documentclass[11pt]{amsart}
\usepackage[top=1.5in, bottom=1.5in, left=1in, right=1in]{geometry}  
\usepackage{indentfirst}
\usepackage[pagebackref]{hyperref}

\usepackage{amsmath, amsfonts, amssymb,amsthm,mathrsfs}
\usepackage{graphicx,fancyhdr}
\usepackage{anyfontsize}
\usepackage{xcolor}
\usepackage{mathptmx} 
\theoremstyle{plain}
\newcommand{\thistheoremname}{}
\newtheorem*{genericthm*}{\thistheoremname}
\newenvironment{namedthm*}[1]{\renewcommand{\thistheoremname}{#1}%
	\begin{genericthm*}}
	{\end{genericthm*}}
\usepackage{empheq}
\usepackage{xcolor}
\usepackage{hyperref}
\hypersetup{
	colorlinks,
	linkcolor={red!50!black},
	citecolor={blue!62!black},
	urlcolor={blue!80!black}
}
\usepackage{enumitem}
\usepackage{cases}
\title[A SMT for Non-Constant Holomorphic Curves Intersecting a Fermat-Type Curve]{\textbf{A Second Main Theorem for Non-Constant Holomorphic Curves Intersecting a Fermat-Type Curve}}
\author{Nguyen Anh Tuan, Tran Thi Thanh Thao, Dinh Khanh Tam,\\ Nguyen Van Trieu, Nguyen Huu Thanh Nhan}
\address{Department of Mathematics, University of Education, Hue University, 34 Le Loi St., Hue City, Vietnam}
\email{natuan@dhsphue.edu.vn}
\date{}
\newtheorem{lemma}{Lemma}
\theoremstyle{plain}
\theoremstyle{remark}

\DeclareMathOperator{\ord}{ord}
\begin{document}
  \begin{abstract}
    In this paper, we prove a Second Main Theorem type inequality for holomorphic curves into the complex projective plane intersecting a Fermat-type curve.
  \end{abstract}
  \maketitle
  
  %Section 1:
  \vspace{-0.5cm}
  \section{INTRODUCTION}
  \vspace{-0.3cm}
  \noindent
  The value distribution theory, initiated by R. Nevanlinna in the 1920s \cite{Nevanlinna1925}, describes how often a holomorphic curve intersects a given family of divisors in a complex manifold. We first briefly recall several standard notations in Nevanlinna theory, which will be used throughout this paper. For an effective divisor $E = \sum_{i\in\mathbb{N}}\alpha_i P_i$ on $\mathbb{C}$, where $\alpha_i$ are positive integers and $P_i \in \mathbb{C}$, and for $k\in\mathbb{N}\cup\{\infty\}$, we define the \textit{$k$-truncated degree} of $E$ on the disk $\Delta_r = \{z \in \mathbb{C}: |z| < r\}$ by
\begin{equation*}
    n^{[k]}(r, E) = \sum_{\substack{P_i \in \Delta_r}} \min\{\alpha_i, k\}.
\end{equation*}
Then, the \textit{$k$-truncated counting function} of $E$ is defined by integrating the $k$-truncated degree as follows
\begin{equation*}
    N^{[k]}(r, E) = \int_1^r \frac{n^{[k]}(t, E)}{t} dt\eqno\scriptstyle{(r>1)}.
\end{equation*}
When $k = \infty$, we write $n(r,E)$ and $N(r,E)$ instead of $n^{[\infty]}(r,E)$ and $N^{[\infty]}(r,E)$, respectively.

Let $f:\mathbb{C} \to \mathbb{CP}^n$ be a holomorphic curve with a reduced representation $$f = [f_0: f_1: \ldots : f_n],$$ where $f_0, f_1, \ldots, f_n$ are entire functions on $\mathbb{C}$ without common zeros. Let $$D = \{Q = 0\}$$ be a hypersurface in $\mathbb{CP}^n$ defined by a homogeneous polynomial $Q$ of degree $d$. If $f(\mathbb{C})\not\subset D$, then the pull-back $f^*D = \sum_{a\in\mathbb{C}} \ord_a(Q(f)) \cdot a$ is an effective divisor on $\mathbb{C}$. Thus, we can define the $k$-truncated counting function of $f$ with respect to $D$ by
\begin{equation*}
    N_f^{[k]}(r, D) := N^{[k]}(r, f^*D),
\end{equation*}
which measures the frequency of the intersections of the image of $f$ with $D$ up to multiplicity $k$. Next, the \textit{proximity function} of $f$ with respect to $D$ is defined by
\begin{equation*}
    m_f(r, D) := \frac{1}{2\pi}\int_0^{2\pi} \log \frac{\|f(re^{i\theta})\|_{\max}^d\|Q\|_{\max}}{|Q(f(re^{i\theta}))|} d\theta,
\end{equation*}
where $\|f(z)\|_{\max} = \max\{|f_0(z)|, |f_1(z)|, \ldots, |f_n(z)|\}$ and $\|Q\|_{\max}$ is the maximum of the modulus of the coefficients of $Q$. Finally, the \textit{characteristic function} of $f$ is defined by
\begin{equation*}
    T_f(r) := \frac{1}{2\pi}\int_0^{2\pi} \log \left\|f(re^{i\theta})\right\|_{\max} d\theta \eqno\scriptstyle{(r>1)}.
\end{equation*}

Nevanlinna theory consists of two fundamental results by comparing the three above functions. The first one is a reformulation of Jensen's formula, which is called the \textit{First Main Theorem}.
\begin{namedthm*}{First Main Theorem \cite{noguchi2013nevanlinna}}
    Let $f:\mathbb{C} \to \mathbb{CP}^n$ be a holomorphic curve and $D$ be a hypersurface in $\mathbb{CP}^n$ of degree $d$ such that $f(\mathbb{C})\not\subset D$. Then, we have
    \begin{equation*}
        m_f(r, D) + N_f(r, D) = d\,T_f(r) + O(1)
    \end{equation*}
    for any $r>1$. Together with the fact that $m_f(r, D) \geq 0$, the following estimate holds:
    \begin{equation*}
        N_f(r, D) \leq d\, T_f(r) + O(1).
    \end{equation*}
\end{namedthm*}
While the First Main Theorem (FMT) establishes a lower bound for the characteristic function in terms of the counting function, the converse, known as the Second Main Theorem, is significantly more intricate and remains an open problem in its general form. The classical version is stated under the assumption that the holomorphic curve $f$ is linearly non-degenerate, i.e., the image of $f$ is not contained in any hyperplane of $\mathbb{CP}^n$, and the family of $q>n+1$ hyperplanes $\{H_i\}_{i=1}^q$ are in general position, i.e. $\bigcap_{i\in I} H_i = \varnothing$ for any subset $I\subset \{1, 2, \ldots, q\}$ with cardinality $|I| = n+1$.
\begin{namedthm*}{Cartan's Second Main Theorem \cite{Cartan1933}}
    Let $f:\mathbb{C} \to \mathbb{CP}^n$ be a linearly non-degenerate holomorphic curve and let $\{H_i\}_{i=1}^q$ be a family of hyperplanes in $\mathbb{CP}^n$ in general position with $q > n+1$. Then,
    \begin{equation*}
        (q - n - 1) T_f(r) \leq \sum_{i=1}^q N_f^{[n]}(r, H_i) + S_f(r)\qquad \|,
    \end{equation*}
    where the notation $\|$ means that the inequality holds for all $r>1$ outside a set of finite Lebesgue measure, and $S_f(r)$ is a small error term satisfying 
    \begin{equation*}
        \liminf_{r\to\infty} \frac{S_f(r)}{T_f(r)} = 0.
    \end{equation*}
\end{namedthm*}
In the case $n=1$, Cartan recovered the classical Second Main Theorem of Nevanlinna for meromorphic functions \cite{Nevanlinna1925}. Lately, many generalizations of the Second Main Theorem have been established \cite{ru2004defect,shiffman1977holomorphic}. For further details on Nevanlinna theory, we refer the readers to \cite{noguchi2013nevanlinna,Ru2021}.

From the geometrical viewpoint, Second Main Theorems are closely related to the hyperbolicity of complex manifolds. Recall that a complex manifold $X$ is said to be \textit{hyperbolic}, in the sense of Brody, if it admits no non-constant holomorphic curves. In the case $n = 1$, the little Picard theorem states that the complement of $3$ points in $\mathbb{CP}^1$ is hyperbolic. In this sense, the Nevanlinna's Second Main Theorem is considered as a quantitative version of the little Picard theorem. In higher dimensions, Cartan's Second Main Theorem quantifies the Borel's result.

The classical approach to the Second Main Theorem relies heavily on the use of Wronskian techniques, which demands at least $n+2$ targets in $\mathbb{CP}^n$. When the number of targets is less than $n+2$, Wronskian techniques are not applicable and the problem becomes more challenging. In the few-target setting, several results about hyperbolicity have been established by Demailly \cite{demailly2012hyperbolic} and Noguchi \cite{noguchi2013nevanlinna}. Following this line of research, we establish a variety of Second Main Theorem type inequalities for non-constant holomorphic curves intersecting a specific Fermat-type curve of degree $d\geq 11$ in $\mathbb{CP}^2$. Our main result is stated as follows:
\begin{namedthm*}{Main Theorem}
    Let $\mathcal{C}\subset\mathbb{CP}^2$ be the curve defined by the homogeneous polynomial
    \begin{equation*}
        z_0^d + z_1^{d-2}(z_1^2+\varepsilon_0 z_2^2) + z_2^{d-2}(\varepsilon_1z_0^2+\varepsilon_2z_1^2+z_2^2)=0,
    \end{equation*}
    where $d\geq 11$ and $\varepsilon_0,\varepsilon_1,\varepsilon_2$ satisfy the following conditions:
    \begin{enumerate}[label = (\roman*)]
        \item $\varepsilon_0,\varepsilon_1,\varepsilon_2\neq 0$;
        \item $\varepsilon_0\varepsilon_2 -1 \neq 0$;
        \item $\varepsilon_1+\frac{d}{2}\left(\frac{2}{d-2}\right)^{\frac{d-2}{d}}\neq 0$;
        \item The resultant $\mathrm{Res}(P(z),Q(z))\neq 0$ where $P(z) = dz^{d-1} + (d-2)\varepsilon_0 z^{d-3} + 2 \varepsilon_2 z$ and $Q(z) = 2 \varepsilon_0 z^{d-2} + (d-2)\varepsilon_2 z^2 + d$.
    \end{enumerate}
    Then, for any \textbf{non-constant} holomorphic curve $f:\mathbb{C}\rightarrow\mathbb{CP}^2$ whose image is not contained in $\mathcal{C}$, we have the Second Main Theorem type inequality
    \begin{equation*}
        (d-10)\,T_f(r) \leq N^{[2]}_f(r,\mathcal{C}) + S_f(r)\qquad\|.
    \end{equation*}
\end{namedthm*}
\noindent The above conditions imposed on $\varepsilon_0, \varepsilon_1, \varepsilon_2$ are required to ensure the hyperbolicity of curves derived from $\mathscr{C}$, which is crucial for the proof of the Main Theorem. 

The rest of this paper is organized as follows: In Section 2, we prove that the genus of some plane curves derived from $\mathcal{C}$, under the suitable conditions, is greater than $1$. These results imply the hyperbolicity of these curves. Then, in Section 3, we establish the Second Main Theorem type inequality for non-constant holomorphic curves intersecting $\mathcal{C}$.
  %Section 2:
  \vspace{-0.5cm}
  \section{GENUS OF SOME PLANE CURVES}
  \vspace{-0.3cm}
  \noindent
  Before proving the main result of this paper, we compute the genus of some plane curves which will be used later. As a corollary, these curves are hyperbolic since the genus of each curve is greater than $1$.
\begin{lemma}[El-Goul \cite{ElGoul1997}]
    \label{theorem: genus 1}
    Let $C$ be a plane curve given by the equation
    \begin{equation*}
        z_0^d + z_2^{d-2}(\varepsilon_1z_0^2+\varepsilon_2z_1^2+z_2^2)=0,
    \end{equation*}
    where $d\geq 11$, $\varepsilon_1,\varepsilon_2\neq 0$ and $\varepsilon_1+\frac{d}{2}\left(\frac{2}{d-2}\right)^{\frac{d-2}{d}}\neq 0$. Let $g$ be the genus of $C$. Then, $g \geq 2$.
\end{lemma}
\begin{proof}
    The proof is given in \cite[Lemma~6.2]{ElGoul1997}.
\end{proof}
\begin{lemma}[El-Goul \cite{goul1996algebraic}]
    \label{theorem: genus 2}
    Let $C$ be a plane curve given by the equation
    \begin{equation*}
        z_0^d + z_1^{d-2}(z_1^2+\varepsilon_0 z_2^2) = 0,
    \end{equation*}
    where $d\geq 11$ and $\varepsilon_0 \neq 0$. Let $g$ be the genus of $C$. Then, $g \geq 2$.
\end{lemma}
\begin{proof}
    The proof is given in \cite[Lemma 4.3]{goul1996algebraic}.
\end{proof}
\begin{lemma}
    \label{theorem: genus 3}
    Let $C$ be a plane curve given by the equation
    \begin{equation*}
        z_1^{d-2}(z_1^2+\varepsilon_0 z_2^2) + z_2^{d-2}(\varepsilon_1z_0^2+\varepsilon_2z_1^2+z_2^2)=0,
    \end{equation*}
    where $d\geq 11$, $\varepsilon_0,\varepsilon_1,\varepsilon_2\neq 0$ and $\mathrm{Res}(P(x),Q(x)) \neq 0$, where $P(x) = dx^{d-1} + (d-2)\varepsilon_0 x^{d-3} + 2 \varepsilon_2 x$ and $Q(x) = 2 \varepsilon_0 x^{d-2} + (d-2)\varepsilon_2 x^2 + d$. Let $g$ be the genus of $C$. Then,
    \begin{equation*}
        g \geq \frac{d-2}{2} \geq 2.
    \end{equation*}
\end{lemma}
\begin{proof}
    To find singular points of $C$, we consider the system:
  \begin{empheq}[left = \empheqlbrace]{align}
    & 2 \varepsilon_1 z_0 z_2^{d-2} = 0,\label{Ciii:dz_0}\\
    & dz_1^{d-1} + (d-2)\varepsilon_0z_1^{d-3}z_2^2 + 2 \varepsilon_2 z_1 z_2^{d-2} = 0,\label{Ciii:dz_1}\\
    & 2\varepsilon_0 z_2 z_1^{d-2} + (d-2) z_2^{d-3} (\varepsilon_1 z_0^2 + \varepsilon_2 z_1^2 + z_2^2) + 2 z_2^{d-1} = 0.\label{Ciii:dz_2}
  \end{empheq}
  The equation \eqref{Ciii:dz_0} implies that either $z_2=0$ or $z_0=0$. If $z_2=0$, \eqref{Ciii:dz_1} implies that $z_1=0$. Thus, $[1:0:0]$ is a singular point of $C$. If $z_2 \neq 0$, then $z_0=0$. Then, \eqref{Ciii:dz_1} and \eqref{Ciii:dz_2} imply that
  \begin{equation*}
    \begin{cases}
      dx^{d-1} + (d-2)\varepsilon_0 x^{d-3} + 2 \varepsilon_2 x = 0, \\
      2 \varepsilon_0 x^{d-2} + (d-2)\varepsilon_2 x^2 + d = 0,
    \end{cases}
  \end{equation*}
  where $x = \frac{z_1}{z_2}$. Denote $P(x) = dx^{d-1} + (d-2)\varepsilon_0 x^{d-3} + 2 \varepsilon_2 x$ and $Q(x) = 2 \varepsilon_0 x^{d-2} + (d-2)\varepsilon_2 x^2 + d$. Under the assumption that
  \begin{equation*}
    \mathrm{Res}(P(x),Q(x)) \neq 0,
  \end{equation*}
  these polynomials have no common roots \cite[Corollary 8.4]{lang2012algebra}, which implies that the system has no solutions. 
  Therefore, the only singular point of $C$ is $[1:0:0]$. In the inhomogeneous coordinate $(X,Y) = \left(\frac{z_1}{z_0}, \frac{z_2}{z_0}\right)$, the curve $C$ is locally given by the equation
  \begin{equation*}
    X^{d-2}(X^2 + \varepsilon_0 Y^2) + Y^{d-2}(\varepsilon_1 + \varepsilon_2 X^2 + Y^2) = 0.
  \end{equation*}
  By Taylor expansion, the curve $C$ near the point $(0,0)$ is given by
  \begin{equation*}
    X^d + \varepsilon_1 Y^{d-2} = 0.
  \end{equation*}
  If $d$ is odd then ${C}$ is irreducible. In the case that $d$ is even, we claim that $C$ is also irreducible. On the contrary, we suppose that the polynomial $X^d + \varepsilon_0X^{d-2}Y^2 + \varepsilon_1 Y^{d-2} + \varepsilon_2 X^2 Y^{d-2} + Y^d$ is reducible. Let $Z = \frac{X}{Y}$ and $T = \frac{1}{Y}$, the assumption yields that the polynomial
  \begin{equation}
    \label{eq : d_even_reducible}
    h(Z,T) = Z^d + \varepsilon_0 Z^{d-2} + \varepsilon_2 Z^2 + 1 + \varepsilon_1T^2
  \end{equation}
  is reducible. Since $h(Z,T)$ is a polynomial of degree $2$ in $T$ and $\mathbb{C}[Z]$ is a unique factorization domain, it can only be factored as
  \begin{equation}
    \label{eq : d_even_reducible_factorization}
    h(Z,T) = \varepsilon_1(T + a(Z))(T + b(Z)),
  \end{equation}
  where $a(Z), b(Z) \in \mathbb{C}[Z]$. By comparing coefficients in \eqref{eq : d_even_reducible} and \eqref{eq : d_even_reducible_factorization}, we have
  \begin{equation*}
    \begin{cases}
      a(Z) + b(Z) = 0, \\
      \varepsilon_1a(Z)b(Z) = Z^d + \varepsilon_0 Z^{d-2} + \varepsilon_2 Z^2 + 1,
    \end{cases}
  \end{equation*}
  which implies that $-\varepsilon_1 a(Z)^2 = Z^d + \varepsilon_0 Z^{d-2} + \varepsilon_2 Z^2 + 1$. It means that any root of the right-hand side must be a multiple root. However, the condition
  \begin{equation*}
    \mathrm{Res}(P(Z), Q(Z)) \neq 0
  \end{equation*}
  implies that 
  \begin{equation*}
    \mathrm{Res}\left(\frac{1}{d}(Z\cdot P(Z) + Q(Z)), P(Z)\right) = \mathrm{Res}(Z^d + \varepsilon_0 Z^{d-2}+ \varepsilon_2 Z^2 + 1, dZ^{d-1} + (d-2)\varepsilon_0 Z^{d-3} + 2 \varepsilon_2Z) \neq 0.
  \end{equation*}
  It yields that the polynomial $Z^d + \varepsilon_0 Z^{d-2} + \varepsilon_2 Z^2 + 1$ has no multiple roots. This is a contradiction. Thus, $C$ is irreducible.
  Therefore, by \cite[Proposition~4.1]{goul1996algebraic}, its genus satisfies
  \begin{align*}
    g & = \frac{(d-1)(d-2)}{2} - \frac{(d-1)(d-3)+\gcd(d,d-2)-1}{2} \\
      & \geq \frac{(d-1)(d-2)}{2} - \frac{(d-1)(d-3)+1}{2} \\
      & = \frac{d-2}{2} \geq 2.
  \end{align*}
\end{proof}
  %Section 3:
  \vspace{-0.5cm}
  \section{A SECOND MAIN THEOREM FOR NON-CONSTANT HOLOMORPHIC CURVES INTERSECTING A FERMAT-TYPE CURVE}
  \vspace{-0.3cm}
  \noindent
  Now, we prove the \textbf{Main Theorem}.
Let $f:\mathbb{C}\to\mathbb{CP}^2$ be a non-constant holomorphic curve and $f = [f_0:f_1:f_2]$ be a reduced representation of $f$ in the homogeneous coordinates $[z_0:z_1:z_2]$. 
Let $g = [g_0:g_1:g_2]:\mathbb{C}\rightarrow\mathbb{CP}^2$ where $g_0=f_0^d$, $g_1=f_1^{d-2}(f_1^2+\varepsilon_0f_2^2)$, $g_2=f_2^{d-2}(\varepsilon_1f_0^2+\varepsilon_2f_1^2+f_2^2)$. We claim that $g$ is well-defined. Indeed, suppose that there exists $z\in \mathbb{C}$ such that $g_0(z) = g_1(z) = g_2(z) = 0$. Then, we have the system
\begin{equation*}
    \begin{cases}
        f_0^d = 0,\\
        f_1^{d-2}(f_1^2+\varepsilon_0f_2^2) = 0,\\
        f_2^{d-2}(\varepsilon_1f_0^2+\varepsilon_2f_1^2 + f_2^2) = 0,
    \end{cases}
\end{equation*}
which is equivalent to
\begin{equation*}
    \begin{cases}
        f_0 = 0,\\
        f_1 = 0 \text{ or } f_1^2 + \varepsilon_0f_2^2 = 0,\\
        f_2 = 0 \text{ or } \varepsilon_2f_1^2 + f_2^2 = 0.
    \end{cases}
\end{equation*}
If $f_1 = 0$, then $f_2 = 0$ (and vice versa), which contradicts the assumption that $f = [f_0:f_1:f_2]$ is a reduced representation. Hence, $f_1\not= 0$ and $f_2 \neq 0$. Then, the system is equivalent to
\begin{equation*}
    \begin{cases}
        f_0 = 0,\\
        x^2 + \varepsilon_0 = 0,\\
        \varepsilon_2 x^2 + 1 = 0,
    \end{cases}
\end{equation*}
where $x = \frac{f_1}{f_2}$. Under our assumption that 
\begin{equation*}
    \varepsilon_0\varepsilon_2 - 1 \neq 0,
\end{equation*}
the system has no solutions. Therefore, $g$ is well-defined.

Next, we prove that $g$ is non-constant. On the contrary, suppose that $g$ is constant. In the case that $g_0 \equiv 0$, if $g_2 \equiv 0$ then $f_2^{d-2}(\varepsilon_2f_1^2 + f_2^2) \equiv 0$, which implies that $f_2 \equiv 0$ or $f_2 \equiv \lambda f_1$ where $\lambda$ satisfies $\lambda^2 + \varepsilon_2 = 0$. In both cases, $f$ is constant, which is a contradiction. If $g_2 \not\equiv 0$, then there exists a constant $c$ such that $g_1 \equiv c g_2$, which is equivalent to
\begin{equation*}
    f_1^{d-2}(f_1^2+\varepsilon_0f_2^2) \equiv cf_2^{d-2}(\varepsilon_2f_1^2+f_2^2).
\end{equation*}
It implies $f_1 \equiv \lambda f_2$ where $\lambda$ is a constant satisfying $\lambda^{d} + \varepsilon_0\lambda^{d-2} - c(\varepsilon_2\lambda^2 + 1) = 0$. Thus, $f \equiv [0:\lambda:1]$ is constant, a contradiction. If $f_0^d\not\equiv 0$, then there exist constants $c_1,c_2$ such that
\begin{equation*}
    \begin{cases}
        f_1^{d-2}(f_1^2+\varepsilon_0f_2^2) &\equiv c_1f_0^d,\\
        f_2^{d-2}(\varepsilon_1f_0^2+\varepsilon_2f_1^2+f_2^2) &\equiv c_2f_0^d.
    \end{cases} 
\end{equation*}
If $c_1 = 0$, then $f_1\equiv 0$ or $f_1 \equiv \lambda f_2$ where $\lambda$ satisfies $\lambda^2 + \varepsilon_0 = 0$. In both cases, together with the second equation, $f$ is constant, which is a contradiction. If $c_1\neq 0$, the first equation implies that the image of $f$ is contained in the curve
\begin{equation*}
    C = \{-c_1z_0^d + z_1^{d-2}(z_1^2+\varepsilon_0z_2^2) = 0\}.
\end{equation*}
However, by Lemma \ref{theorem: genus 2}, this curve is hyperbolic with $d\geq 11$, which contradicts the assumption that $f$ is non-constant. Thus, the claim is proved; one has $T_g(r) = dT_f(r) + O(1)$.

Let $\{H_i\}_{0\leq i\leq 3}$ be four projective lines in $\mathbb{CP}^2$ where
\begin{align*}
    H_i&=\{z_i=0\}\qquad(0\leq i\leq 2),\\
    H_{3}&=\{\sum_{i=0}^{2}z_i=0\},
\end{align*}
which is in general position in $\mathbb{CP}^2$. First, we consider the case that $g$ is linearly non-degenerate. Applying Cartan's Second Main Theorem for $g$ and the family $\{H_i\}_{0\leq i\leq 3}$, we have
\begin{equation} 
    \label{ineq: non-degenerate}
    T_g(r) \leq \sum_{i=0}^{3} N^{[2]}_g(r,H_i) + S_g(r) \qquad \|.
\end{equation}
One has estimates:
\begin{align*}
    N^{[2]}_g(r,H_0) &= \min\{2,d\}\,N^{[1]}_g(r,H_0) = \min\{2,d\}\,N^{[1]}_f(r,H_0) \\
    &\leq 2N_f(r,H_0)\overset{FMT}{\leq} 2T_f(r)+O(1),
\end{align*}
and
\begin{align*}
    N^{[2]}_g(r,H_1) &= \min\{2,d-2\}\,N^{[1]}_f(r,H_1) + N^{[2]}_f(r,R_1) \\
    &\leq 2N_f(r,H_1)+N_f(r,R_1)\\
    &\hspace{-6pt}\overset{FMT}{\leq} 4T_f(r)+O(1),
\end{align*}
where $R_1$ is the conic defined by the equation $z_1^2+\varepsilon_0z_2^2=0$. Similarly, we have
\begin{align*}
    N^{[2]}_g(r,H_2) &= \min\{2,d-2\}\,N^{[1]}_f(r,H_2) + N^{[2]}_f(r,R_2) \\
    &\leq 2N_f(r,H_2)+N_f(r,R_2)\\
    &\hspace{-6pt}\overset{FMT}{\leq} 4T_f(r)+O(1),
\end{align*}
where $R_2$ is the conic defined by the equation $\varepsilon_1z_0^2+\varepsilon_2z_1^2+z_2^2=0$. Moreover, we have
\begin{equation*}
    N^{[2]}_g(r,H_3) = N^{[2]}_f(r,\mathcal{C}).
\end{equation*}
Then, the estimate \eqref{ineq: non-degenerate} becomes
\begin{equation*}
    (d-10)T_f(r) \leq N^{[2]}_f(r,\mathcal{C}) + S_f(r) \qquad \|.
\end{equation*}

Next, we consider the case that $g$ is linearly degenerate. Since $f(\mathbb{C})\not\subset\mathscr{C}$, $g(\mathbb{C})$ cannot be contained in the line $H_3$. We consider the case that $g(\mathbb{C})$ belongs to a coordinate line, which leads to the following three subcases:
    \begin{enumerate}[label=\roman*.]
        \item $g(\mathbb{C})\subset H_0\cong\mathbb{CP}^1$. Applying Nevanlinna's Second Main Theorem for $g$ and three distinct points $H_0\cap H_i$ ($1\leq i\leq 3)$, we obtain
    \begin{equation}
        \label{ineq: degenerate H_0}
        T_g(r) \leq \sum_{i=1}^{3}N^{[1]}_g(r,H_i\cap H_0) + S_g(r) \qquad \|.
    \end{equation}
    As above, we have estimates:
    \begin{align*}
        N^{[1]}_g(r,H_1\cap H_0) &\leq N^{[1]}_g(r,H_1) = N^{[1]}_f(r,H_1) + N^{[1]}_f(r,R_1) \\
            &\leq N_f(r,H_1) + N_f(r,R_1)\overset{FMT}{\leq} 3T_f(r)+O(1),\\
        N^{[1]}_g(r,H_2\cap H_0) &\leq N^{[1]}_g(r,H_2) = N^{[1]}_f(r,H_2) + N^{[1]}_f(r,R_2) \\
            &\leq N_f(r,H_2) + N_f(r,R_2)\overset{FMT}{\leq} 3T_f(r)+O(1),\\
        N^{[1]}_g(r,H_3\cap H_0) &\leq N^{[1]}_g(r,H_3) = N^{[1]}_f(r,\mathcal{C}).
    \end{align*}
    Thus, the estimate \eqref{ineq: degenerate H_0} becomes
    \begin{equation*}
        (d-6)T_f(r) \leq N^{[1]}_f(r,\mathcal{C}) + S_f(r)\qquad\|.
    \end{equation*}
    \item $g(\mathbb{C})\subset H_1\cong\mathbb{CP}^1$\label{subcase H1}. Applying Nevanlinna's Second Main Theorem for $g$ and three distinct points $H_1\cap H_i$ ($i\neq 1$), we obtain
    \begin{equation}
        \label{ineq: degenerate H_1}
        T_g(r) \leq \sum_{i\neq 1}N^{[1]}_g(r,H_i\cap H_1) + S_g(r)\qquad\|.
    \end{equation}
    Similarly, we have:
    \begin{align*}
        N^{[1]}_g(r,H_0\cap H_1) &\leq N^{[1]}_g(r,H_0) = N^{[1]}_f(r,H_0)\\
        &\leq N_f(r,H_0)\overset{FMT}{\leq} T_f(r)+O(1),\\
        N^{[1]}_g(r,H_2\cap H_1) &\leq N^{[1]}_g(r,H_2) \leq 3T_f(r)+O(1),\\
        N^{[1]}_g(r,H_3\cap H_1) &\leq N^{[1]}_g(r,H_3) = N^{[1]}_f(r,\mathcal{C}).
    \end{align*}
    Thus, the estimate \eqref{ineq: degenerate H_1} becomes
    \begin{equation*}
        (d-4)T_f(r) \leq N^{[1]}_f(r,\mathcal{C}) + S_f(r) \qquad \|.
    \end{equation*}
    Then, we obtain the desired estimate.
    \item $g(\mathbb{C}) \subset H_2 \cong \mathbb{CP}^1$. By symmetry, this subcase can be treated similarly to the subcase  \ref{subcase H1}.
\end{enumerate}

  Finally, we consider the case that $g(\mathbb{C})\subset L$ which is not a coordinate line. We claim that $L$ intersects $\{H_i\}_{0\leq i \leq 3}$ at more than two distinct points. On the contrary, suppose that $L$ intersects $\{H_i\}_{0\leq i \leq 2}$ at exactly two distinct points. Then, we have three subcases:
\begin{enumerate}[label = \roman*.]
  \item\label{subcase:3i}
    $L$ passes through $H_0 \cap H_2$ and $H_1 \cap H_3$.  Hence, $L=\{z_0+z_2=0\}.$ Since $g(\mathbb{C}) \subset L$ then $f(\mathbb{C}) \subset C$ defined by the polynomial $$z_0^d+z_2^{d-2}(\varepsilon_1z_0^2+\varepsilon_2z_1^2+z_2^2) = 0.$$ By Lemma \ref{theorem: genus 1}, this curve is hyperbolic under our assumption. It implies that $f$ must be constant, which is a contradiction.
    
  \item $L$ passes through $H_0 \cap H_1$ and $H_2 \cap H_3$.  Hence, $L=\{z_0+z_1 = 0\}.$ Since $g(\mathbb{C}) \subset L$, we have $f(\mathbb{C}) \subset C$ defined by the polynomial $$ z_0^d + z_1^{d-2}(z_1^2 + \varepsilon_0z_2^2 ) = 0. $$ By Lemma \ref{theorem: genus 2}, this curve is hyperbolic under our assumption. It implies that $f$ must be constant, which is a contradiction.
  
  \item $L$ passes through $H_0 \cap H_3$ and $H_1 \cap H_2$.  Hence, $L=\{z_1+z_2 = 0\}$. Since $g(\mathbb{C}) \subset L$ then $f(\mathbb{C}) \subset C$ defined by the polynomial $$z_1^{d-2}(z_1^2 + \varepsilon_0z_2^2)+z_2^{d-2}(\varepsilon_1z_0^2+\varepsilon_2z_1^2+z_2^2) = 0.$$ By Lemma \ref{theorem: genus 3}, this curve is hyperbolic under our assumption. It implies that $f$ must be constant, which is a contradiction.
\end{enumerate}

In the case that $L$ intersects $\{H_i\}_{0\leq i \leq 2}$ at three distinct points, by applying Nevanlinna’s Second Main Theorem for \( g \) and the three points \( L \cap H_i \)
(\( 0 \le i \le 2 \)), we obtain
\begin{align*}
  \label{ineq: degenerate L}
  T_g(r) &\le \sum_{i=0}^{2} N_g^{[1]}(r, L \cap H_i) + S_g(r) \qquad \|\\
  &\leq \sum_{i=0}^{2} N_g^{[1]}(r, H_i) + S_g(r) \qquad \|,
\end{align*}
which is equivalent to
\begin{equation}
    dT_f(r) \leq 7T_f(r) + S_f(r) \qquad \|.
\end{equation}
Since $d\geq 11$, this is a contradiction. Hence, this subcase cannot occur.

  In the case that $L$ intersects $\{H_i\}_{0\leq i \leq 2}$ at exactly two distinct points and intersects $H_3$ at another point, there are three possible subcases to consider:
\begin{enumerate}[label = \roman*.]
    \item \label{subcase:ii} $L$ passes through $H_0\cap H_1$ and intersects $H_2$, $H_3$ at two distinct points. Applying Nevanlinna's Second Main Theorem for $g$ and the three points $H_0\cap H_1$, $L\cap H_2$, $L\cap H_3$, we have
    \begin{align*}
        T_g(r) &\leq N^{[1]}_g(r,H_0\cap H_1) + N^{[1]}_g(r,L\cap H_2) + N^{[1]}_g(r,L\cap H_3) + S_g(r) \qquad \| \\
         &\leq N_g^{[1]}(r,H_0) + N_g^{[1]}(r,H_2) + N_f^{[1]}(r,\mathcal{C}) + S_g(r) \qquad \|,
    \end{align*}
    which is equivalent to
    \begin{equation*}
        dT_f(r) \leq 4T_f(r) + N_f^{[1]}(r,\mathcal{C}) + S_f(r)\qquad \|.
    \end{equation*}
    Then, the desired estimate follows.    
    \item $L$ passes through $H_0\cap H_2$ and intersects $H_1$, $H_3$ at two distinct points. By symmetry, this subcase can be treated similarly to the subcase \ref{subcase:ii}.
    \item $L$ passes through $H_1 \cap H_2$ and intersects $H_0,H_3$ at two distinct points. Then, by applying Nevanlinna's Second Main Theorem for $g$ and the three points $H_1\cap H_2$, $L\cap H_0$, $L\cap H_3$, we have
    \begin{align*}
        T_g(r) &\leq N^{[1]}_g(r,H_1\cap H_2) + N^{[1]}_g(r,L\cap H_0) + N^{[1]}_g(r,L\cap H_3) + S_g(r) \qquad \|\\
         & \leq N_g^{[1]}(r,H_1) + N_g^{[1]}(r,H_0) + N_f^{[1]}(r,\mathcal{C}) + S_g(r) \qquad \|,
    \end{align*}
    which is equivalent to
    \begin{equation*}
        dT_f(r) \leq 4T_f(r) + N_f^{[1]}(r,\mathcal{C}) + S_f(r)\qquad \|.
    \end{equation*}

    The proof is complete.
\end{enumerate}

  %---------------------------------------
  \section*{Acknowledgement}
  \vspace{-0.3cm}
  \noindent This research is funded by the University of Education, Hue University under grant number NCTBSV.T.25 TN - 101 - 02. The authors would like to express their sincere gratitude to Dr. Dinh Tuan Huynh for his valuable guidance and support throughout the research.
  %---------------------------------------
  \bibliographystyle{plain}
  \bibliography{references}
  %---------------------------------------
\end{document}